\documentclass[11pt]{article}
\usepackage[utf8]{inputenc}
\usepackage[T1]{fontenc}
\usepackage{amsthm,amsmath,amssymb}
\usepackage{geometry}
\usepackage{enumitem}
\usepackage{microtype}
\usepackage{xcolor}
\usepackage[colorlinks=true,linkcolor=blue!70!black,citecolor=green!50!black,
            urlcolor=blue!60!black]{hyperref}

\newtheorem{theorem}{Theorem}[section]
\newtheorem{lemma}[theorem]{Lemma}
\newtheorem{proposition}[theorem]{Proposition}
\newtheorem{corollary}[theorem]{Corollary}

\theoremstyle{definition}
\newtheorem{definition}[theorem]{Definition}

\theoremstyle{remark}
\newtheorem{remark}[theorem]{Remark}
\newtheorem{openproblem}[theorem]{Open Problem}

\newcommand{\R}{\mathbb{R}}
\newcommand{\Rb}{\overline{\mathbb{R}}}

\title{The Extended Real Line with Reentry:\\
A Compact Path-Connected US-not-KC Space}
\author{Dami\'{a}n Rafael Lattenero\\[4pt]
\small\texttt{damian.lattenero@gmail.com}}
\date{October 2026}

\begin{document}
\maketitle

\begin{abstract}
We construct what is, to the author's knowledge, the first compact
path-connected US-not-KC space.
The Extended Real Line with Reentry (ERI) identifies $-\infty$, $0$ and
$+\infty$ to a single point $\ast$ of the extended real line and
requires every neighborhood of $\ast$ to have dense preimage under the
quotient map.
ERI is compact, path-connected, sober, and Baire; it is $T_1$, US, and
has unique limits even for transfinite sequences, but not KC, not
weakly Hausdorff, and not Hausdorff; sequential
convergence is explicit: $\hat{x}_n \to \ast$ iff $(x_n)$ has no
accumulation point in $\R \setminus \{0\}$.
The compact US-not-KC spaces previously on record are totally
disconnected or fail path-connectedness; a one-point extension of
$(0,1)$, analyzed alongside, achieves the same separation and is not
homeomorphic to ERI.
\end{abstract}


\section{Introduction}\label{sec:intro}

A space is \emph{US} if every convergent sequence has a unique limit,
and \emph{KC} if every compact subset is closed
(Wilansky~\cite{bib:wilansky}); $T_2 \Rightarrow \mathrm{KC}
\Rightarrow \mathrm{US} \Rightarrow T_1$, all strict.
Clontz~\cite{bib:clontz} refined this chain to
\begin{equation}\label{eq:clontz}
T_2 \;\Rightarrow\; k_1\mathrm{H} \;\Rightarrow\; \mathrm{KC}
\;\Rightarrow\; \mathrm{wH} \;\Rightarrow\; k_2\mathrm{H}
\;\Rightarrow\; \mathrm{US} \;\Rightarrow\; T_1,
\end{equation}
where $\mathrm{wH}$ (weakly Hausdorff) requires every continuous image
of a compact Hausdorff space to be closed, and $k_2\mathrm{H}$
($k_2$-Hausdorff, from Clontz--Williams~\cite{bib:clontzwilliams})
requires the kernel $\{(a,b) : f(a) = f(b)\}$ of every continuous map
$f\colon K \to X$ from compact Hausdorff $K$ to be closed in $K^2$.
Clontz and Williams further interpolate between $k_2\mathrm{H}$ and
US: $T_2 \Rightarrow \mathrm{UR} \Rightarrow \mathrm{UOK} \Rightarrow
\mathrm{UCR} \Rightarrow \mathrm{US}$ and $k_2\mathrm{H} \Rightarrow
\mathrm{UOK}$, where UR (resp.\ UCR) requires unique limits of
arbitrary (resp.\ continuous) transfinite
sequences~\cite{bib:clontzwilliams}.
A first-countable space is Hausdorff iff it is
US~\cite{bib:goreham}, so every US-not-KC space must fail first
countability somewhere.
The compact US-not-KC spaces on record---the square of the one-point
compactification of $\mathbb{Q}$ (S31), $\omega_1{+}1$ with doubled
endpoint (S37), and the one-point compactification of the Arens--Fort
space (S165)~\cite{bib:clontz,bib:clontzwilliams}---are totally
disconnected, and no compact connected US-not-KC example appears in
these surveys or in pi-Base~\cite{bib:pibase}.
A compact \emph{connected} US-not-KC space---$[0,1]$ with an added
point whose neighborhoods have dense preimage---was posted on Math
StackExchange in 2022~\cite{bib:mse}; it is not
path-connected.\footnote{If $\gamma$ were a path from the added point
$\ast$ into $[0,1]$, let $t^\ast := \max \gamma^{-1}(\ast) < 1$
(which exists since $\{\ast\}$ is closed and $\gamma(1) \in [0,1]$)
and pick $t_n \downarrow t^\ast$ with $\gamma(t_n) \in [0,1]$;
by compactness a subsequence has $\gamma(t_n) \to c \in [0,1]$; then
$F := \{\gamma(t_n)\}_n \cup \{c\}$ is closed nowhere dense in $[0,1]$,
and $\{\ast\} \cup ([0,1] \setminus F)$ is an open neighborhood of
$\ast$ containing no $\gamma(t_n)$, contradicting continuity at
$t^\ast$.}
A connected US-not-KC extension of $\R$ by a point with co-countable
neighborhoods (pi-Base S140) is neither compact nor path-connected.
Van~Douwen's compact anti-Hausdorff Fr\'{e}chet--Urysohn
US-space~\cite{bib:vandouwen} is connected (anti-Hausdorff spaces are
hyperconnected) but KC---indeed every Fr\'{e}chet--Urysohn US space is
KC~\cite{bib:kunzi}---which blocks the usual sequential
constructions.

The \emph{Extended Real Line with Reentry} (ERI) identifies
$\{-\infty, 0, +\infty\}$ to a single point $\ast$ of
$\Rb = [-\infty,+\infty]$ and requires every neighborhood of $\ast$ to
have dense preimage; the density condition destroys Hausdorffness
while preserving $T_1$, US, compactness, and path-connectedness.

\begin{theorem}[Main theorem]\label{thm:main}
ERI is compact, path-connected, locally connected, sober, Baire, and
sequentially compact.
It is $T_1$, US, and even UR, but not KC, not weakly Hausdorff, and not
Hausdorff.
First countability fails only at $\ast$, where
$t(\ast) = \aleph_0$ while $\chi(\ast) \geq \aleph_1$.
Every continuous map into ERI from a first-countable compact Hausdorff
space has closed kernel; the analogous statement for arbitrary
compact Hausdorff domains is open
(Section~\ref{sec:open}).
\end{theorem}

ERI is, to the author's knowledge, the first compact path-connected
US-not-KC space; the space $Z$ of Section~\ref{sec:onepoint},
suggested during the review of an earlier version, is the second.
Like S165, ERI is US and not weakly Hausdorff, and the two are
otherwise far apart (S165 is countable and totally disconnected).
Section~\ref{sec:construction} constructs ERI and proves the basic
properties; Section~\ref{sec:convergence} develops its sequential
theory, including a complete convergence criterion and unique limits
for transfinite sequences;
Section~\ref{sec:hierarchy} pins down US, not-KC, and not-wH.
Section~\ref{sec:fmq} generalizes: the density condition defines a
\emph{filter-modified quotient} (FMQ) over any compact Hausdorff base
without isolated points, and the $k_2$-Hausdorff property is
established there for first-countable compact domains.
Section~\ref{sec:onepoint} analyzes a one-point extension of $(0,1)$
that reaches the same separation; it is not homeomorphic to ERI.
Theorem~\ref{thm:main} summarizes
Sections~\ref{sec:construction}--\ref{sec:hierarchy};
Theorems~\ref{thm:k2H} and~\ref{prop:onepoint} supply the
first-countable-domain kernel property for both spaces.
A broader program around this construction is developed in an extended
companion version~\cite{bib:companion}.

\section{The ERI space}\label{sec:construction}

\begin{definition}[ERI]\label{def:ERI}
Let $\Rb = [-\infty,+\infty]$ with its order topology, let $\sim$ be
the equivalence relation whose sole non-singleton class is
$A := \{-\infty, 0, +\infty\}$, and let
$q\colon \Rb \to X_0 := \Rb/{\sim}$ be the projection.
Write $\ast := q(A)$ and $\hat{x} := q(x)$, so that
$q|_{\R \setminus \{0\}}$ bijects onto $X_0 \setminus \{\ast\}$.
The \emph{ERI topology} is
\[
\tau \;:=\; \bigl\{\, U \subseteq X_0 :
  \text{(a) } q^{-1}(U) \text{ is open in } \Rb,\;
  \text{and (b) } \ast \in U \Rightarrow \overline{q^{-1}(U)} = \Rb
  \,\bigr\},
\]
and $X := (X_0, \tau)$.
For a subset $F \subseteq \R \setminus \{0\}$ (so $F$ avoids
$\pm\infty$) that is closed in $\Rb$ and nowhere dense in $\Rb$, we
write $W_F := q(\Rb \setminus F)$.
\end{definition}

Condition (a) alone defines the standard quotient topology $\tau_q$,
which is Hausdorff (quotient of a compact Hausdorff space by a closed
equivalence relation~\cite[2.4.15]{bib:engelking}); condition (b) is
the single modification that produces everything below.

\begin{lemma}\label{lem:basic}
$\tau$ is a topology, $q\colon \Rb \to X$ is continuous, and the open
neighborhoods of $\ast$ are exactly the sets $W_F$ with $F \subseteq
\R \setminus \{0\}$ closed and nowhere dense in $\Rb$ (with
$q^{-1}(W_F) = \Rb \setminus F$).
\end{lemma}

\begin{proof}
Unions are immediate since a dense summand makes the union dense
(condition (b) of Definition~\ref{def:ERI} is upward closed in the
preimage), and continuity is condition (a).
For finite intersections with $\ast$ in both sets, the preimages are
dense and open, and dense open sets are stable under finite
intersection~\cite[1.3.5]{bib:engelking}.
If $U \ni \ast$ is open, then $F_U := \Rb \setminus q^{-1}(U)$ is
closed (its complement is open), has empty interior (its complement is
dense), and avoids $\{0\}$ (as $q^{-1}(U)$ is saturated and meets $A$,
it contains $A$); and $U = W_{F_U}$ because $q^{-1}(U) = \Rb \setminus
F_U$ is saturated.
Conversely, if $F \subseteq \R\setminus\{0\}$ is closed nowhere-dense
in $\Rb$, then $\Rb \setminus F \supseteq A$ is saturated, open, and
dense, so $W_F$ is open and $q^{-1}(W_F) = \Rb \setminus F$.
\end{proof}

\begin{lemma}\label{lem:kip}
Every open neighborhood of $\ast$ meets every nonempty open set; hence
$\ast \in \overline{O}$ for every nonempty open $O \subseteq X$, and
every open neighborhood of $\ast$ is dense in $X$.
\end{lemma}

\begin{proof}
$q^{-1}(U)$ is dense by (b) and $q^{-1}(V)$ is nonempty open by (a), so
the preimages meet.
\end{proof}

\begin{proposition}\label{prop:basics}
$X$ is $T_1$, not Hausdorff ($\ast$ cannot be separated from any other
point), compact, connected, path-connected, and locally connected.
Moreover $X \setminus \{\ast\} \cong \R \setminus \{0\}$ as a subspace,
and $X$ admits no coarser Hausdorff topology.
\end{proposition}

\begin{proof}
$X \setminus \{\hat{x}\}$ has preimage $\Rb \setminus \{x\}$, open and
dense; $X \setminus \{\ast\}$ has preimage $\R \setminus \{0\}$, open,
and condition (b) is vacuous.
So singletons are closed.
Non-Hausdorffness is Lemma~\ref{lem:kip}.
Compactness: $\tau \subseteq \tau_q$ and $(X_0,\tau_q)$ is a compact
quotient of $\Rb$.
Connectedness: a separation $X = U \sqcup V$ with $\ast \in U$
contradicts Lemma~\ref{lem:kip}.
Path-connectedness: for $\hat{a}, \hat{b} \in X$ pick representatives
$a,b \in \Rb$ (with $0$ representing $\ast$); then
$q \circ ((1-t)a + tb)$ is a path from $\hat{a}$ to $\hat{b}$.
Local connectedness: the basic neighborhoods $W_F$ are connected---if
$W_F = U \sqcup V$ with $\ast \in U$ and $U, V$ open in $W_F$, then
$U$ and $V$ are open in $X$ (as $W_F$ is), and Lemma~\ref{lem:kip}
gives $U \cap V \neq \varnothing$; and away from $\ast$ the subspace
$X \setminus \{\ast\} \cong \R \setminus \{0\}$ is locally connected.
The subspace claim is (a) with $\ast \notin U$.
A coarser Hausdorff topology would separate $\ast$ from $\hat{x}$ by
disjoint $\tau$-open sets, contradicting Lemma~\ref{lem:kip}.
\end{proof}

\begin{remark}[Figure eight]\label{rem:figure8}
Under the quotient map, $X \setminus \{\ast\} \cong (-\infty,0) \cup
(0,+\infty)$: ERI is a two-loop figure eight whose vertex $\ast$ is
topologically inseparable from every other point---the ``reentry'' of
the title: a path leaving through $+\infty$ or $-\infty$ returns at
$0$, all three points being $\ast$.
The standard quotient $(X_0, \tau_q)$ is the ordinary Hausdorff figure
eight, the identification of the extended line's endpoints being a
classical example~\cite[Ex.~3.1.10]{bib:richter}; condition (b) is the
only difference.
\end{remark}

\begin{theorem}\label{thm:sober}
$X$ is sober and Baire.
\end{theorem}

\begin{proof}
Since $X$ is $T_1$, a closed set can be irreducible only if it is a
singleton, and each singleton has a unique generic point; sobriety
therefore reduces to showing that every closed $F$ with $|F| \geq 2$
is reducible.
If $F$ contains two points $\hat{x} \neq \hat{y}$ other than $\ast$,
separate them by disjoint open $U_x, U_y$ inside the metrizable open
subspace $X \setminus \{\ast\} \cong \R \setminus \{0\}$ (so $U_x$,
$U_y$ are open in $X$ and avoid $\ast$); then $F \setminus U_x$ and
$F \setminus U_y$ are proper closed subsets of $F$ covering it.
If $F = \{\ast, \hat{x}\}$, it is the union of two singletons.
Either way $F$ is reducible.

Baire: let $(U_n)$ be dense open in $X$.
Each $q^{-1}(U_n)$ is open in $\Rb$, and dense: if $\ast \in U_n$ this
is condition (b); otherwise $q^{-1}(U_n)$ is dense in $\R\setminus\{0\}$
(as $q|_{\R\setminus\{0\}}$ is a homeomorphism and $U_n$ is dense in
the subspace $X \setminus \{\ast\}$, whose complement $\{\ast\}$ is
nowhere dense), hence dense in $\Rb$.
$\Rb$ is compact Hausdorff, hence Baire, so
$\bigcap_n q^{-1}(U_n)$ is dense, and its image meets the preimage of
every nonempty open $V \subseteq X$.
\end{proof}

\section{Sequential convergence at \texorpdfstring{$\ast$}{*}}\label{sec:convergence}

Throughout, $(x_n)$ is a sequence in $\R \setminus \{0\}$ and
$\hat{x}_n := q(x_n)$.

\begin{theorem}[Convergence criterion]\label{thm:seqconv}
$\hat{x}_n \to \ast$ in $X$ if and only if $(x_n)$ has no accumulation
point in $\R \setminus \{0\}$.
\end{theorem}

\begin{proof}
($\Rightarrow$) If a subsequence $x_{n_k} \to t \in \R \setminus \{0\}$,
then $C := \{x_{n_k}\}_k \cup \{t\}$ is closed, countable, and avoids
$0$, hence nowhere dense; so $W_C$ is a neighborhood of $\ast$
(Lemma~\ref{lem:basic}) that contains no $\hat{x}_{n_k}$.
($\Leftarrow$) For $U \ni \ast$ open, the complement
$F := \Rb \setminus q^{-1}(U)$ is a compact subset of
$\R\setminus\{0\}$; if infinitely many $x_n$ lay in $F$,
Bolzano--Weierstrass would give an accumulation point there.
So eventually $x_n \notin F$, i.e.\ $\hat{x}_n \in U$.
\end{proof}

Different subsequences may converge to different members of
$\{-\infty, 0, +\infty\}$: e.g.\ $x_{2k} = k$, $x_{2k+1} = 1/k$ has
$\hat{x}_n \to \ast$ without $(x_n)$ converging in $\Rb$.

\begin{proposition}\label{prop:unique}
Convergent sequences in $X$ have unique limits.
\end{proposition}

\begin{proof}
Let $\hat{x}_n \to L_1$ and $\hat{x}_n \to L_2$.
If both limits differ from $\ast$, then eventually
$\hat{x}_n \in X \setminus \{\ast\}$ (as $\{\ast\}$ is closed), and the
Hausdorff subspace $X \setminus \{\ast\} \cong \R \setminus \{0\}$
forces $L_1 = L_2$.
If $L_1 = \hat{c} \neq \ast$ while $L_2 = \ast$: eventually
$\hat{x}_n \in X \setminus \{\ast\}$ with $x_n \to c$, so $c$ is an
accumulation point of $(x_n)$ in $\R \setminus \{0\}$, and
Theorem~\ref{thm:seqconv} forbids $\hat{x}_n \to \ast$.
The case $L_1 = L_2 = \ast$ is trivial.
\end{proof}

\begin{proposition}\label{prop:countability}
$\ast$ has no countable neighborhood base, yet
$t(\ast, X) = \aleph_0$; every other point is first-countable.
Consequently $X$ is not a sequential
space~\cite{bib:franklin}: $q([1,2])$ is sequentially closed but not
closed.
\end{proposition}

\begin{proof}
Given countably many neighborhoods $U_n \ni \ast$, write
$U_n = W_{F_n}$ (Lemma~\ref{lem:basic}) and let
$D := \bigcap_n (\Rb \setminus F_n)$, dense by the Baire property of
$\Rb$.
Pick $t \in D \cap (\R \setminus \{0\})$: then $W_{\{t\}} \ni \ast$ is
open and no $U_n \subseteq W_{\{t\}}$ (as $t \in q^{-1}(U_n)$), so the
$U_n$ are not a base.
So $\chi(\ast, X) \geq \aleph_1$, and every other point is
first-countable (the intervals $q((x - 1/n,\, x + 1/n))$, $n > 1/|x|$,
suffice).
Not sequential: any sequence in $[1,2]$ has a Bolzano--Weierstrass
accumulation point in $[1,2] \subset \R\setminus\{0\}$, so by
Theorem~\ref{thm:seqconv} no sequence in $q([1,2])$ converges to
$\ast$, and limits in $X\setminus\{\ast\} \cong \R\setminus\{0\}$ stay
in $q([1,2])$; yet $q([1,2])$ is not closed
(Theorem~\ref{thm:notKC} below).
Countable tightness: if $\ast \in \overline{S}$ with
$S \subseteq X \setminus \{\ast\}$, let $D$ be the lift of $S$ to
$\R\setminus\{0\}$.
No closed nowhere-dense $F \subset \R\setminus\{0\}$ contains $D$,
or else $W_F$ would be a neighborhood of $\ast$ missing $S$.
Since $\Rb$ is second-countable, $D$ has a countable subset $D_0$ dense
in $D$, so $\overline{D_0}^{\,\Rb} \supseteq \overline{D}^{\,\Rb}
\supseteq D$; as above, no closed nowhere-dense set contains $D_0$, so
$\ast \in \overline{q(D_0)}$.
\end{proof}

\begin{remark}[Cardinal invariants at $\ast$]\label{rem:cardinals}
Thus $t(\ast, X) = \aleph_0 < \aleph_1 \leq \chi(\ast, X)$: ERI is a
compact, countably tight, non-sequential, sequentially compact space.
\end{remark}

\begin{proposition}\label{prop:seqcpt}
$X$ is sequentially compact.
\end{proposition}

\begin{proof}
By sequential compactness of $\Rb$, some subsequence $x_{n_k} \to L \in
\Rb$ (constant subsequences at $\ast$ being trivial).
If $L \in \R\setminus\{0\}$ then $\hat{x}_{n_k} \to \hat{L}$ in the
subspace $X \setminus \{\ast\}$.
If $L \in A$, then $L$ is the only accumulation point of
$(x_{n_k})$ in $\Rb$, so Theorem~\ref{thm:seqconv} gives
$\hat{x}_{n_k} \to \ast$.
\end{proof}

\begin{theorem}[Unique radial convergence]\label{thm:UR}
Let $\lambda$ be a limit ordinal and $(\hat{x}_\alpha)_{\alpha<\lambda}$
a transfinite sequence in $X$, that is, a function from $\lambda$ to
$X$ in the sense of Clontz--Williams~\cite{bib:clontzwilliams}; a point
$p$ is a \emph{limit} of the sequence if every neighborhood of $p$
contains a final segment of it.
Then the sequence has at most one limit: $X$ is UR.
\end{theorem}

\begin{proof}
Let the sequence converge to $p$ and $q$.
If neither is $\ast$, both lie in $X \setminus \{\ast\}$, an open set
containing a final segment of the sequence and homeomorphic to the
Hausdorff space $\R \setminus \{0\}$, which forces $p = q$.
The case $p = q = \ast$ is trivial, so suppose $p = \hat{x}$ with $x \in
\R \setminus \{0\}$ and $q = \ast$ (the symmetric case is identical).
Since $X \setminus \{\ast\}$ is an open neighborhood of $\hat{x}$, there
is $\alpha_0 < \lambda$ with $\hat{x}_\alpha \in X \setminus \{\ast\}$
for $\alpha \geq \alpha_0$; write $x_\alpha \in \R \setminus \{0\}$ for
those indices.

Suppose first $\operatorname{cf}(\lambda) > \omega$, and fix $n_0$ with
$1/n_0 < |x|$.
For each $n \geq n_0$ the set $q((x - 1/n,\, x + 1/n))$ is an open
neighborhood of $\hat{x}$, so there is $\alpha_n < \lambda$ with
$|x_\alpha - x| < 1/n$ whenever $\alpha \geq \alpha_n$.
Every countable subset of $\lambda$ is bounded below $\lambda$, so
$\alpha^\ast := \sup_{n \geq n_0} \max(\alpha_0, \alpha_n) < \lambda$,
and $\alpha \geq \alpha^\ast$ forces $|x_\alpha - x| < 1/n$ for all
$n \geq n_0$, that is, $x_\alpha = x$: the sequence is eventually
constant at $\hat{x}$.
But $W_{\{x\}}$ is an open neighborhood of $\ast$
(Lemma~\ref{lem:basic}) omitting $\hat{x}$, contradicting convergence
to $\ast$.

Suppose now $\operatorname{cf}(\lambda) = \omega$.
Choose a strictly increasing cofinal sequence $\beta_n \uparrow \lambda$
with $\beta_n \geq \max(\alpha_0, \alpha_n)$ for $n \geq n_0$; then
$x_{\beta_n} \to x$ in $\R \setminus \{0\}$.
Convergence of the whole transfinite sequence to $\ast$ means that for
every $F \subseteq \R \setminus \{0\}$ closed and nowhere dense in
$\Rb$ there is $\gamma_F < \lambda$ with $\hat{x}_\alpha \in W_F$ for
$\alpha \geq \gamma_F$; cofinality of $(\beta_n)$ gives
$\hat{x}_{\beta_n} \in W_F$ for all large $n$, so
$\hat{x}_{\beta_n} \to \ast$.
But $(x_{\beta_n})$ has the accumulation point $x \in \R \setminus
\{0\}$, which Theorem~\ref{thm:seqconv} forbids.
\end{proof}

\section{US, not KC, not weakly Hausdorff}\label{sec:hierarchy}

\begin{theorem}\label{thm:notKC}
$X$ is US but not KC: $q([1,2])$ is compact and not closed.
Moreover $X$ is SC (every convergent sequence with its limit is a
closed set).
\end{theorem}

\begin{proof}
US is Proposition~\ref{prop:unique}.
$q([1,2])$ is compact (a continuous image of a compact set), and
$\ast \notin q([1,2])$ while
$q^{-1}\bigl(X \setminus q([1,2])\bigr) = \Rb \setminus [1,2]$ is not
dense (its complement has interior $(1,2)$), so
$X \setminus q([1,2])$ is not open.
SC: let $y_n \to L$ and $S := \{y_n\} \cup \{L\}$.
If $y_n = \ast$ for infinitely many $n$, the constant subsequence
forces $L = \ast$; either way, discarding the (at most finitely many)
terms equal to $\ast$ changes $S$ by a finite set and we may write
$y_n = q(x_n)$ with $x_n \in \R \setminus \{0\}$.
If $L = \hat{c}$ with $c \in \R\setminus\{0\}$, then $x_n \to c$;
the set $C := \{c\} \cup \{x_n\}$ is closed nowhere-dense in
$\R\setminus\{0\}$ and $q^{-1}(X \setminus S) = \Rb \setminus C$ is
open and dense.
If $L = \ast$, then $\ast \in S$ makes condition (b) vacuous, and
$A \cup \{x_n\}$ is closed in $\Rb$ because by
Theorem~\ref{thm:seqconv} all accumulation points of $(x_n)$ lie in
$A$.
\end{proof}

\begin{proposition}\label{prop:notwH}
$X$ is not weakly Hausdorff.
\end{proposition}

\begin{proof}
The embedding $[1,2] \hookrightarrow \Rb \xrightarrow{q} X$ is
continuous, has compact Hausdorff domain, and image $q([1,2])$, which
is not closed by Theorem~\ref{thm:notKC}.
\end{proof}

\begin{remark}
Since first-countable US spaces are Hausdorff~\cite{bib:goreham}, the
failure of first countability at $\ast$
(Proposition~\ref{prop:countability}) is the precise obstruction: away
from $\ast$ the space is metrizable.
Every continuous map from $X$ to a Hausdorff space with at least two
points is constant---separating $f(\hat{x})$ from $f(\ast)$ by disjoint
open sets contradicts Lemma~\ref{lem:kip}---so $X$ is not regular, not
completely regular, and has trivial \v{C}ech cohomology; see~\cite{bib:companion}
for elaboration.
\end{remark}

\section{Filter-modified quotients and the \texorpdfstring{$k_2$}{k2}-Hausdorff theorem}\label{sec:fmq}

The construction depends only on a filter-like family of large
subsets.
The collapse underlying ERI is classical
(Remark~\ref{rem:figure8}); the density modifier and its use in
separation-axiom counterexamples are, to the author's knowledge, new.

\begin{definition}[FMQ]\label{def:fmq}
Let $(Y,\sigma)$ be a space, $\varnothing \neq A \subseteq Y$,
$q\colon Y \to Y/A$ the collapse of $A$ to $\ast$.
A family $\mathcal{F} \subseteq \mathcal{P}(Y)$ is an
\emph{admissible modifier} if $Y \in \mathcal{F}$, $\mathcal{F}$ is
upward closed, and $U_1 \cap U_2 \in \mathcal{F}$ whenever
$U_1, U_2 \in \mathcal{F} \cap \sigma$.
The \emph{filter-modified quotient} $\mathrm{FMQ}(Y,A,\mathcal{F})$ is
$Y/A$ with the topology $\tau_{\mathcal{F}}$ of sets $U$ with
$q^{-1}(U)$ open and ($\ast \in U \Rightarrow q^{-1}(U) \in
\mathcal{F}$).
The \emph{density modifier} is $\mathcal{D}_Y := \{D \subseteq Y :
\overline{D} = Y\}$.
\end{definition}

$\mathcal{D}_Y$ is admissible on any space (finite intersections of
dense open sets are dense), and
$X = \mathrm{FMQ}(\Rb, \{-\infty,0,+\infty\}, \mathcal{D}_{\Rb})$.

\begin{theorem}[Generalized ERI]\label{thm:general}
Let $Y$ be compact Hausdorff without isolated points and $A \subseteq
Y$ finite nonempty.
Then $\mathrm{FMQ}(Y, A, \mathcal{D}_Y)$ is compact, $T_1$, connected,
US, not KC, and not Hausdorff.
\end{theorem}

\begin{proof}[Proof sketch]
$\tau_{\mathcal{F}}$ is a topology by admissibility (as in
Lemma~\ref{lem:basic}) and is coarser than the quotient
topology, which gives compactness.
$T_1$: for $y \notin A$, the complement of $\hat{y}$ has preimage
$Y \setminus \{y\}$, open and dense (no isolated points), and
$\ast$ lies in it, so both conditions hold; $\ast$ is closed since $A$
is closed.
Connectedness is Lemma~\ref{lem:kip} verbatim.
For US and not-Hausdorff, the contrapositive halves of the arguments
in Sections~\ref{sec:convergence}--\ref{sec:hierarchy} transfer: if a
sequence in the quotient converged to two limits, some subsequence
would have lifted values converging in $Y$ to a point $r \in
Y \setminus A$, and the values of that convergent subsequence
together with $r$ form a countable closed set with empty
interior in $Y$ (closed, as the continuous image of the compact
convergent-sequence space; interiorless because a nonempty open
subset of the compact Hausdorff,
isolated-point-free $Y$ is uncountable), which is exactly the
separating set $C$ of Theorems~\ref{thm:seqconv} and~\ref{thm:notKC}.
Not KC: $q(K)$ is compact and not closed for any compact
$K \subset Y \setminus A$ with nonempty interior.
Full details appear in the extended companion
version~\cite{bib:companion}, including the extension to nonempty
closed nowhere-dense collapse sets $A$.
\end{proof}

We now prove the $k_2$-Hausdorff property for first-countable domains
(the last separation axiom asserted in Theorem~\ref{thm:main}).

\begin{lemma}[Closed fibers]\label{lem:closedfibers}
Let $Y$ be compact Hausdorff without isolated points, $A$ finite,
$X = \mathrm{FMQ}(Y,A,\mathcal{F})$ with
$\mathcal{F} \supseteq \mathcal{F}_{\mathrm{com}} := \{D : Y\setminus D
\text{ meager}\}$ and $\mathcal{F} \subseteq \mathcal{D}_Y$.
Let $f\colon K \to X$ be continuous, $K$ compact Hausdorff,
$N := f^{-1}(\ast)$, and
$g := (q|_{Y\setminus A})^{-1} \circ f\colon K \setminus N \to
Y \setminus A$.
Then $g^{-1}(F)$ is closed in $K$ (not merely in $K \setminus N$) for
every $F \subset Y \setminus A$ closed and nowhere dense in $Y$.
\end{lemma}

\begin{proof}
$W_F := q(Y \setminus F)$ is open with $\ast \in W_F$ (as
$Y \setminus F$ is open, contains $A$, and is comeager hence in
$\mathcal{F}$), and
$K \setminus f^{-1}(W_F) = \{k \in K\setminus N : g(k) \in F\}
= g^{-1}(F)$.
\end{proof}

\begin{proposition}[Accumulation structure]\label{prop:accum}
In the setting of Lemma~\textup{\ref{lem:closedfibers}}, assume
moreover that every point of $Y \setminus A$ has a countable
neighborhood base in $Y$.
Let $(k_n)$ be a \emph{sequence} in $K \setminus N$ with
$k_n \to a \in N$.
Then no subsequence of $(g(k_n))$ converges in $Y \setminus A$.
(The same argument applies verbatim to any net whose index set
admits a countable cofinal subset.)
\end{proposition}

\begin{proof}
Suppose $g(k_{n_j}) \to z \in Y \setminus A$, and fix a decreasing
base $V_1 \supseteq V_2 \supseteq \cdots$ at $z$.
By Lemma~\ref{lem:closedfibers} applied to the closed nowhere-dense
$\{v\}$, for each $v \in Y \setminus A$ the set $g^{-1}(\{v\})$ is
closed in $K$ and misses $a$, so $g(k_n) = v$ for only finitely many
$n$.
Choose $j_1 < j_2 < \cdots$ with
$g(k_{n_{j_i}}) \in V_i \setminus \{z\}$, possible with all values
distinct since each forbidden value excludes only finitely many
indices.
Then $C := \{g(k_{n_{j_i}})\}_i \cup \{z\}$ is countable and closed
in $Y$ (it is the continuous image of the compact space
$\{0\} \cup \{1/i\}$ under $1/i \mapsto g(k_{n_{j_i}})$,
$0 \mapsto z$),
and being a countable union of nowhere-dense singletons in the Baire
space $Y$ it is closed meager, hence nowhere dense.
Lemma~\ref{lem:closedfibers} gives $g^{-1}(C)$ closed in $K$ and
containing $k_{n_{j_i}} \to a$, hence
$a \in g^{-1}(C) \subseteq K \setminus N$---contradicting $a \in N$.
\end{proof}

\begin{theorem}[$k_2$-Hausdorff for first-countable domains]\label{thm:k2H}
In the setting of Proposition~\textup{\ref{prop:accum}}, assume
further that $X$ is $T_1$.
Then for every first-countable compact Hausdorff $K$ (in particular
every metrizable $K$) and every continuous $f\colon K \to X$, the
kernel $\ker(f)$ is closed in $K \times K$.
In particular, ERI is $k_2$-Hausdorff in the sense
of~\cite{bib:clontzwilliams} for all first-countable compact Hausdorff
domains.
\end{theorem}

\begin{proof}
In a first-countable space a set is closed iff it is sequentially
closed, so let $(k_n, \ell_n) \to (a, b)$ be a sequence in
$\ker(f)$; we show $f(a) = f(b)$.
By continuity, $f(k_n) = f(\ell_n)$ converges to both $f(a)$ and
$f(b)$.
If neither $f(a)$ nor $f(b)$ is $\ast$, then the sequence is
eventually in the Hausdorff subspace $X \setminus \{\ast\}$, which has
unique limits.
If both are $\ast$, we are done.
Otherwise say $f(a) = \ast$, $f(b) = \hat{r}$, $r \in Y \setminus A$.
Eventually $f(k_n) = f(\ell_n) \neq \ast$ (as $\{\ast\}$ is closed and
the limit along $\ell_n$ is $\hat{r} \neq \ast$), so both sequences
live in $K \setminus N$ where $g$ is defined, and
$f(k_n) = f(\ell_n)$ gives $g(k_n) = g(\ell_n) \to g(b) = r \in
Y \setminus A$.
But $k_n \to a \in N$, so Proposition~\ref{prop:accum} forbids any
subsequence of $(g(k_n))$ from converging in $Y \setminus A$.
\end{proof}

\begin{remark}[The referee's observation]\label{rem:referee}
The $k_2$-Hausdorff property of ERI was first suggested by a referee
of \emph{Topology and its Applications}, who sketched a direct
argument along the following lines: with $f(k) = \hat{r} \neq
f(\ell)$, $r \in \R \setminus \{0\}$, the set
$K \setminus f^{-1}(\ast)$ is locally compact (open in the compact
Hausdorff $K$), and a compact neighborhood of $k$ mapping into
$q((r - s, r+s))$, $0 < s < |r|$, should separate the two points.
Completing this separation in general is precisely what the
accumulation structure above is for; Theorem~\ref{thm:k2H} carries
out the idea for first-countable domains, and the general case is
Open Problem~\ref{op:main}.
We record the suggestion with thanks.
\end{remark}

\begin{remark}\label{rem:betaN}
The countable-base hypothesis in Proposition~\ref{prop:accum} is used
to make the value set countable, hence meager; the countable
cofinality of the index set is what lets the extracted subsequence
converge to $a$.
For nets of uncountable cofinality neither device applies, which is
exactly the content of Open Problem~\ref{op:main}.
\end{remark}

\begin{remark}[The spectrum below $k_2$H]\label{rem:spectrum}
In the refinement of Clontz--Williams, $T_2 \Rightarrow \mathrm{UR}
\Rightarrow \mathrm{UOK} \Rightarrow \mathrm{UCR} \Rightarrow
\mathrm{US}$ and $k_2\mathrm{H} \Rightarrow
\mathrm{UOK}$~\cite{bib:clontzwilliams}.
By Theorem~\ref{thm:UR}, ERI is UR, so it satisfies every property of
that spectrum lying between US and $k_2\mathrm{H}$, and the only
separation axiom there still undecided for ERI is $k_2\mathrm{H}$
itself.
The known compact witnesses at this level fail path-connectedness:
$\omega_1{+}1$ with doubled endpoint (S37) is US but not UCR, and the
$\beta\omega$-based twin space of~\cite{bib:clontzwilliams} is UR but
not $k_2\mathrm{H}$, with every path constant (it maps continuously
onto $\beta\omega$ with fibers of size at most two).
Thus, if Open Problem~\ref{op:main}(1) has a negative answer, ERI is a
compact, path-connected, UR-not-$k_2\mathrm{H}$ space, the
path-connected counterpart of their example.
As a byproduct, since a pseudo-C-radial UCR space is
KC~\cite{bib:clontzwilliams}, ERI (which is UCR and not KC) is not
pseudo-C-radial.
\end{remark}

\section{A one-point extension of \texorpdfstring{$(0,1)$}{(0,1)}}\label{sec:onepoint}

The following construction was suggested by a referee during the review
of an earlier version of this work (see Acknowledgments); it reaches
the same separation as ERI.

\begin{definition}\label{def:Z}
Let $Z := (0,1) \cup \{\ast\}$, where $(0,1)$ carries its usual
topology, and $U \ni \ast$ is open iff $(0,1) \setminus U$ is a compact
nowhere-dense subset of $(0,1)$.
\end{definition}

This is a topology: finite unions of compact nowhere-dense sets are
compact nowhere-dense (Baire category in $(0,1)$), and arbitrary
intersections of compact sets are compact with empty interior
(subsets of nowhere-dense sets).
Note that $Z$ is itself a filter-modified quotient:
\begin{equation}\label{eq:ZisFMQ}
Z \;\cong\; \mathrm{FMQ}\bigl([0,1],\, \{0,1\},\,
\mathcal{D}_{[0,1]}\bigr),
\end{equation}
the unit interval with its two endpoints collapsed---the smallest
interval base carrying the density construction.
The general results of Section~\ref{sec:fmq} therefore apply to $Z$
verbatim.

\begin{theorem}\label{prop:onepoint}
$Z$ is compact, $T_1$, sober, US, UR, not KC, path-connected, not weakly
Hausdorff, and $k_2$-Hausdorff for first-countable domains
(by~\eqref{eq:ZisFMQ} and Theorem~\ref{thm:k2H}).
Moreover $z_n \to \ast$ in $Z$ iff $(z_n)$ has no accumulation point in
$(0,1)$.
\end{theorem}

\begin{proof}[Proof sketch]
The arguments of Sections~\ref{sec:construction}--\ref{sec:fmq}
transfer with $(0,1)$ in the role of $\Rb \setminus A$ and
$\{0,1\}$ (viewed as the two ends of $(0,1)$) in the role of $A$.
Compactness: an open cover member $U \ni \ast$ leaves a compact
remainder.
$T_1$: singletons in $(0,1)$ are compact nowhere-dense.
Not KC and not wH: $[1/3, 2/3]$ is compact with nonempty interior,
hence not closed, and it is a continuous image of $[1/3,2/3]$.
The convergence criterion is Theorem~\ref{thm:seqconv} with
Bolzano--Weierstrass in $(0,1)$; US, sobriety, and local connectedness
then follow as in Propositions~\ref{prop:unique}
and~\ref{prop:basics} and Theorem~\ref{thm:sober}.
Path-connectedness: for $c \in (0,1)$, the path $\gamma(0) = \ast$,
$\gamma(t) = tc$ ($t > 0$) is continuous at $0$ because every
neighborhood of $\ast$ has closed complement $C \subset (0,1)$ with
$\inf C > 0$ (vacuously if $C = \emptyset$).
The $k_2$-Hausdorff statement is~\eqref{eq:ZisFMQ} plus
Theorem~\ref{thm:k2H}.
UR: the proof of Theorem~\ref{thm:UR} transfers verbatim with $(0,1)$
in place of $\R \setminus \{0\}$ and the convergence criterion just
stated in place of Theorem~\ref{thm:seqconv} (singletons are compact
nowhere-dense in $(0,1)$).
\end{proof}

\begin{proposition}\label{prop:nohomeo}
$Z$ is not homeomorphic to ERI: the complement of $\ast$ is
disconnected in ERI and connected in $Z$.
\end{proposition}

\begin{proof}
Call $p$ \emph{ubiquitous} if $p \in \overline{O}$ for every nonempty
open $O$.
In both spaces $\ast$ is ubiquitous (Lemma~\ref{lem:kip}; in $Z$,
complements of neighborhoods of $\ast$ have empty interior, so every
neighborhood of $\ast$ meets every nonempty open set), and it is the
only ubiquitous point: for any other point $c$, an open interval at
positive distance from $c$, lying to one side of $c$ (inside $(0,1)$,
resp.\ $\R \setminus \{0\}$), is an open set $O$ with
$c \notin \overline{O}$.
A homeomorphism must therefore carry $\ast$ to $\ast$.
But $X \setminus \{\ast\} \cong \R \setminus \{0\}$ has two connected
components while $Z \setminus \{\ast\} = (0,1)$ is connected.
\end{proof}

\begin{corollary}\label{cor:two}
ERI and $Z$ are non-homeomorphic compact, path-connected, sober
US-not-KC spaces; neither is weakly Hausdorff, both are UR, and both
satisfy the $k_2$-Hausdorff property for first-countable compact
domains.
\end{corollary}

\begin{remark}\label{rem:samefilter}
By Lemma~\ref{lem:basic} and Definition~\ref{def:Z}, the neighborhood
filter of the singular point is, in both spaces, the filter of
complements of compact nowhere-dense sets; by~\eqref{eq:ZisFMQ} both
spaces are density-modified quotients, over $\Rb$ and $[0,1]$
respectively.
The non-homeomorphism (Proposition~\ref{prop:nohomeo}) is what
distinguishes them.
\end{remark}

\section{Final remarks}\label{sec:open}

The density construction supports a broader program, recorded in an
extended companion version~\cite{bib:companion}: a spectrum of
admissible modifiers in which the density modifier is the least
restrictive one preserving US, invariance of the resulting hierarchy
level under larger collapse sets, iteration, and products, and further
examples---among them a Cantor-set variant that is connected but not
path-connected, confirming that path-connectedness depends on the base
minus the collapse locus.

\begin{openproblem}\label{op:main}
\begin{enumerate}[label=\textup{(\arabic*)},nosep]
\item Is ERI $k_2$-Hausdorff---that is, does
Theorem~\ref{thm:k2H} extend to arbitrary compact Hausdorff domains?
The only gap is Proposition~\ref{prop:accum} for nets whose index
sets have uncountable cofinality.
\item Does there exist a compact path-connected US space that is not
$k_2$-Hausdorff?
A negative answer to (1) would settle this affirmatively, with ERI
itself UR (Theorem~\ref{thm:UR})---the path-connected counterpart of
the compact UR-not-$k_2$H example
of~\cite{bib:clontzwilliams}; if (1) is
positive, such an example would have to escape the hypotheses of
Theorem~\ref{thm:k2H}---first-countability of the domain or of the
base---and the natural candidate is a filter-modified quotient over
the non-first-countable base $\beta\mathbb{N} \setminus \mathbb{N}$.
\end{enumerate}
\end{openproblem}

\paragraph{Acknowledgments.}
The author thanks Steven Clontz for maintaining pi-Base and for
helpful correspondence; an anonymous referee of \emph{Topology and its
Applications}, whose report on an earlier version supplied both the
$k_2$-Hausdorff observation recorded as Remark~\ref{rem:referee} and
the one-point extension studied in Section~\ref{sec:onepoint}; and
Mart\'{\i}n Escard\'{o} for comments on an early version of this work,
which identified the classical antecedent of the underlying quotient
in Richter's book~\cite{bib:richter}.

\paragraph{Note on methodology.}
This work was developed between 2023 and 2026 with the extensive
assistance of large language model tools (ChatGPT, Claude, Gemini,
DeepSeek), used for drafting, proof exploration, and literature
search; the construction of ERI and the decision to pursue the
compact path-connected US-not-KC question were the author's.
The tools also participated in exploring proof strategies; final
arguments were adopted only after line-by-line verification.
All definitions and proofs were verified line by line by the author,
who is solely responsible for the mathematical content, and every
reference was checked against its primary source.
No funding was received for this work; the author declares no
conflicts of interest.


\end{document}